\documentclass[11pt]{amsart}
\usepackage{wasysym,amssymb,eufrak,indentfirst,graphicx,cite,amsthm,color}
\usepackage[bookmarksnumbered, colorlinks, linktocpage, plainpages]{hyperref}
\usepackage{epsfig}
\usepackage{float}
\usepackage{indentfirst}
\newtheorem{theorem}{Theorem}[section]
\newtheorem{corollary}[theorem]{Corollary}
\newtheorem{lemma}[theorem]{Lemma}
\newtheorem{proposition}[theorem]{Proposition}
\theoremstyle{definition}

\theoremstyle{remark}
\newtheorem{remark}[theorem]{Remark}

\numberwithin{equation}{section}
\begin{document}

\title[Schwarz lemma at the boundary of strongly pseudoconvex domains]
{Boundary Schwarz lemma for holomorphic self-mappings of strongly pseudoconvex domains}

\author[X. P. Wang and G. B. Ren]{Xieping Wang and Guangbin Ren}
\address{Department of Mathematics, University of Science and
Technology of China, Hefei 230026, China}

\email{pwx$\symbol{64}$mail.ustc.edu.cn; rengb$\symbol{64}$ustc.edu.cn}

\thanks{The second author is supported by the NNSF  of China (11371337), RFDP (20123402110068).}


\begin{abstract}
In this paper, we generalize a recent work of Liu et al.  from the open unit ball $\mathbb B^n\subset\mathbb C^n$
to more general bounded strongly pseudoconvex domains with $C^2$ boundary.
It turns out that part of the main result in this paper is in some certain sense just a part of results in a work of  Bracci and Zaitsev. However, the proofs are significantly different: the argument in this paper  involves a simple growth estimate for the Carath\'{e}odory metric near the boundary of $C^2$ domains and the well-known Graham's estimate  on the boundary behavior of the Carath\'{e}odory metric on strongly pseudoconvex domains, while Bracci and Zaitsev use other arguments.

\end{abstract}

\maketitle

\section{Introduction}

The Schwarz lemma as one of the most influential  results in  complex analysis  puts a great push to the development of several research fields,   such as  geometric function theory, hyperbolic geometry,  complex dynamical systems,  composition operators theory, and theory of quasi-conformal mappings. We refer to \cite{Abate, EJLS} for  a more complete insight on the Schwarz lemma.

As a variant of the classical interior Schwarz lemma,  a new  boundary Schwarz lemma in one complex variable was first independently established by Unkelbach \cite{Unkelbach} and Herzig \cite{Herzig}, and was rediscovered by Osserman \cite{Osserman} more than sixty years later. It turns out to be a powerful tool for solving a variety of problems in complex analysis. Recently, in  \cite{Bracci2}, it together with Lempert's complex geodesic theory was used to study the homogeneous complex Monge-Amp\`{e}re equation with a simple singularity at the boundary of strongly convex domains in $\mathbb C^n$. It was proved that for every bounded strongly convex domain with smooth boundary, there exists a smooth solution to this equation, called the pluricomplex Poisson kernel, of which the sub-level sets being exactly Abate's horopsheres with centers at the boundary. Bracci and Patrizio have succeeded in characterizing biholomorphisms between two bounded strongly convex domains by means of these pluricomplex Poisson kernels; see \cite{Bracci2, Bracci3} for more details.

In complex analysis of several variables, there are also a variety of Schwarz lemmas.
In \cite{Wu}, Wu proved what is now called the Carath\'{e}odory-Cartan-Kaup-Wu theorem,
which generalizes the classical Schwarz lemma for holomorphic functions to
higher dimensions. In \cite{BK}, Burns and Krantz obtained a new Schwarz lemma at the
boundary of strongly pseudoconvex domains, which gives a new rigidity result for holomorphic mappings. In \cite{Huang}, Huang
further strengthened the Burns-Krantz result in the case of strongly convex domains
for holomorphic mappings with an interior fixed point, and
see \cite{BZZ} for the other generalizations of the Burns-Krantz result. Very recently, a new boundary Schwarz lemma was proved by Liu et al. in \cite{LWX} for holomorphic self-mappings of the open unit ball $\mathbb B^n\subseteq\mathbb C^n$, and is essentially a direct consequence of Rudin's generalization of the classical Julia-Wolff-Carath\'{e}odory theorem, see also \cite{Wang-Ren} for a simpler proof and a stronger version. As applications,  it was used to give a new and simple proof of the distortion theorem of determinants for biholomorphic convex mappings on $\mathbb B^n$ and to establish partly the distortion theorem of determinants for biholomorphic starlike mappings in \cite{LWX}. Incidentally, there is also a new boundary Schwarz lemma for slice regular functions, see \cite{WR} for more details.

In this paper, we generalize the recent work of Liu et al. \cite{LWX} from the open unit ball $\mathbb B^n$ to more general bounded strongly pseudoconvex domains with $C^2$ boundary. After the authors finished a preliminary  version of this paper, Bracci brought the paper \cite{Bracci4} into their attention. Part of the main result of this paper is only the ``only if" part of \cite[Propostion 1.1]{Bracci4} in the case where the considered function is a holomorphic self-mapping, with a prescribed regular boundary fixed point, of a bounded strongly pseudoconvex domain $\Omega\subset \subset\mathbb C^n$ with $C^\infty$ boundary replaced by $C^2$ boundary. However, the proofs are significantly different: the argument in this paper  involves a simple growth estimate for the Carath\'{e}odory metric near the boundary of $C^2$ domains and the well-known Graham's estimate  on the boundary behavior of the Carath\'{e}odory metric on strongly pseudoconvex domains, while Bracci and Zaitsev use other arguments. In order to describe precisely what we will prove, we first recall a few standard  notations and necessary terminologies.

Let $\mathbb C^n$ denote the $n$-dimensional complex Hilbert space endowed with the standard Hermitian inner product given by
\begin{equation}\label{Hermitian product}
\langle z, w\rangle=\sum\limits_{j=1}^nz_j\overline{w}_j
\end{equation}
for vectors $z=(z_1, z_2, \ldots,  z_n)^t$, $w=(w_1, w_2, \ldots,  w_n)^t\in\mathbb C^n$. Here the symbol $^t$ stands for the transpose of vectors or matrices. We also denote by $\mathbb B^n$ the open unit ball of $\mathbb C^n$, i.e.
\begin{equation*}
\mathbb B^n:=\big\{z\in\mathbb C^n: |z|<1\big\}.
\end{equation*}
Throughout this paper, we write a point $z\in\mathbb C^n $ as a column vector, i.e. a $n\times1$ matrix.
  Let $\Omega$ be a domain in $\mathbb C^n$. For each holomorphic mapping $f=(f_1, f_2,\ldots,f_n)^t$ from $\Omega$ to $\mathbb C^n $, the derivative of $f$ at a point $z\in \Omega$ is the complex Jacobian matrix of $f$ given by
\begin{equation*}
J_f(z)=\bigg(\frac{\partial f_j}{\partial z_k}(z)\bigg)_{n\times n}.
\end{equation*}

Now the main result in this paper can be stated as follows:
\begin{theorem}\label{main result-Schwarz}
Let $\Omega\subset \mathbb C^n$ be a bounded strongly pseudoconvex domain with $C^2$ boundary and $f:\Omega\rightarrow \Omega$  a holomorphic mapping. Suppose that $f$ extends smoothly past some point $p\in\partial \Omega$ and $f(p)=p$. Then for the eigenvalues $\lambda, \mu_2, \ldots, \mu_n$ $($counted with multiplicities$)$ of $J_f(p)$, the following statements hold:
\begin{enumerate}
  \item [(i)] $\lambda $ is positive and is also an eigenvalue of $\overline{J_f(p)}^t$ such that
  $$\overline{J_f(p)}^t\nu_p=\lambda\nu_p,$$
  %and
%  \begin{equation}\label{Julia-Schwarz ineq}
%  f^{\ast}\Omega_{D,p}\leq \frac1 {\lambda}\Omega_{D,p},
%  \end{equation}
where $\nu_p$ is the unit outward normal vector to $\partial \Omega$ at the point $p$;
  \item [(ii)] $\mu_j\in\mathbb C$ and $|\mu_j|\leq \sqrt{\lambda}$ for  $j=2, 3,\ldots, n$;
  \item [(iii)] For  $j=2, 3,\ldots, n$, there exists $\tau_j\in T^{(1,0)}_p(\partial \Omega)\cap \partial \mathbb B^n$ such that
      $$J_f(p)\tau_j=\mu_j\tau_j;$$
  \item [(iv)] $|\det J_f(p)|\leq\lambda^{\frac{n+1}2},\qquad
|{\rm{tr}} J_f(p)|\leq\lambda+(n-1)\sqrt{\lambda}$.
\end{enumerate}
Moreover, if $f$ has an interior fixed point $z_0\in \Omega$, then $\lambda\geq 1.$
\end{theorem}

Obviously, the Jacobian $J_f(p)$ of $f$ at the point $p\in\partial \Omega$  may be possibly degenerate. A relevant and very interesting result is a boundary version of the open mapping theorem for holomorphic mappings with a prescribed super-regular contact point between two strongly pseudoconvex domains in  $\mathbb C^n$, see \cite{Bracci01} for more details.

The starting point of the proof of the previous theorem is to show  that the real tangent space $T_{p}(\partial \Omega)$ and the complex tangent space $T_{p}^{(1,0)}(\partial \Omega)$ are respectively invariant under the action of $J_f(p)$ as linear transformations of $\mathbb R^{2n}$ and $\mathbb C^n$. This fact easily follows from  a careful consideration of the geometrical information of $f$ at its prescribed boundary fixed point $p\in\partial \Omega$ and the classical Hopf lemma from PDEs. However, a crucial difficulty arises in proving the estimate in the assertion ${\rm{(ii)}}$. The proof of  assertion ${\rm{(ii)}}$ will systematically use the geometric properties of the Carath\'{e}odory metric of bounded strongly pesudoconex domains. It is interesting to notice that the better geometric understanding given by this tool (and the impossibility of using the kind of explicit calculations done in \cite{LWX}) for the ball) yields a proof that is both simpler and clearer than the previous one of \cite[Theorem 3.1]{LWX}.

This paper is organized as follows. In Sect. \ref{Preliminaries}, we first collect some well-known facts on the intrinsic complex  geometry of domains in $\mathbb C^n$ by means of the intrinsic Carath\'{e}odory distance and metric. We then establish a simple growth estimate for the Carath\'{e}odory metric near the boundary of  bounded $C^2$  domains and recall the well-known Graham's estimate  on the boundary behavior of the Carath\'{e}odory metric on strongly pseudoconvex domains, which will play a fundamental role in our argument. Sect. \ref{Proof of main result} is devoted to a proof of  main result of this paper.

\section{Preliminaries}\label{Preliminaries}

In this section, we collect some well-known facts on the intrinsic complex  geometry of domains in $\mathbb C^n$, which are conveniently described using the intrinsic Carath\'{e}odory (pseudo)distance and (pseudo)metric. We refer to \cite{Abate,J-Pflug,Kobayashi} for details and much more on the Carath\'{e}odory distance  and metric on complex manifolds; here we shall just recall what is needed for our purpose. Let $\omega$ denote the \textit{Poincar\'{e} distance} on the open unit disk $\mathbb D\subset\mathbb C$. Let $\Omega\subset\mathbb C^n$ be a bounded domain, the $\textit{Carath\'{e}odory distance}$  $C_{\Omega}: \Omega\times \Omega\rightarrow \mathbb R^+$ of $\Omega$ is defined as
$$C_{\Omega}(z, w):=\sup\Big\{\omega\big(f(z), f(w)\big): f\in {\rm{Hol}} (\Omega, \mathbb D)\Big\}$$
for all $z, w\in \Omega$. It turns out that $C_{\Omega}(z, w)$ is always finite.
  Also, for each bounded domain $\Omega\subset \mathbb C^n$ with complex tangent bundle $T^{(1,0)}\Omega$, the \textit{Carath\'{e}odory metric}
$\mathcal{C}_{\Omega}: T^{(1,0)}\Omega\rightarrow\mathbb R^+$ is defined as
\begin{equation}\label{def-Cara}
\mathcal{C}_{\Omega}(z, v):=\sup\Big\{|df_z(v)|: f\in {\rm{Hol}} (\Omega, \mathbb D),\  f(z)=0\Big\}
\end{equation}
for all $z\in \Omega$ and $v\in T_z^{(1,0)}\Omega$; note that since $df_z(v)\in T_0^{(1,0)}\mathbb D$, $|df_z(v)|$ is the length in the Poincar\'{e} metric of $df_z(v)$ at $0$. For each  $f\in {\rm{Hol}} (\Omega, \mathbb D)$ and $z\in \Omega$, we denote by $\varphi_{f(z)}$ the unique holomorphic automorphism of $\mathbb D$ interchanging $0$ and $f(z)$. Then $g:=\varphi_{f(z)}\circ f\in {\rm{Hol}} (\Omega, \mathbb D)$ is such that $g(z)=0$ and
$$|dg_z(v)|=\big|d\varphi_{f(z)}\circ df_z(v)\big|=\frac{|df_z(v)|}{1-|f(z)|^2}\geq|df_z(v)|.$$
Thus the condition that $f(z)=0$ in the definition $(\ref{def-Cara})$ of Carath\'{e}odory metric is superfluous. It also should be remarked that $\mathcal{C}_{\Omega}$ is always locally Lipschitz, and hence continuous, but $C_{\Omega}$ is in general not the integrated form of $\mathcal{C}_{\Omega}$. Moreover, it is also easy to see that  for each $z\in \Omega$,  $\mathcal{C}_{\Omega}(z, v)$ is subadditive in $v\in T_z^{(1,0)}\Omega$. This simple fact will be used in the proof of Theorem $\ref{main result-Schwarz}$.

The main properties of the Carath\'{e}odory distance and metric are that they are contracted by holomorphic mappings: if $\Omega_1\subset\mathbb C^n$, $\Omega_2\subset\mathbb C^m$ are two bounded domains and $f:\Omega_1\rightarrow \Omega_2$ a holomorphic mapping, then
\begin{equation}\label{Caratheodory-Schwarz0}
C_{\Omega_2}\big(f(z), f(w)\big)\leq C_{\Omega_1}(z, w)
\end{equation}
and
\begin{equation}\label{Caratheodory-Schwarz}
\mathcal{C}_{\Omega_2}\big(f(z), df_z(v)\big)\leq \mathcal{C}_{\Omega_1}(z, v)
\end{equation}
for all $z, w\in \Omega_1$ and $v\in T_z^{(1,0)}\Omega_1$. In particular, biholomorphisms are isometries, and holomorphic self-mappings are $C_{\Omega}$ and $\mathcal{C}_{\Omega}$-nonexpansive.

In Sect. \ref{Proof of main result}, we shall use the following lemma concerning a simple growth estimate for the Carath\'{e}odory metric near the boundary of  bounded $C^2$  domains. To state precisely its content, we need some new notations.
Let $\Omega\subset \subset \mathbb C^n$ be a  domain with $C^2$ boundary and  $\delta(z)=d(z, \partial \Omega)$ the Euclidean distance from $z\in \mathbb C^n$   to the boundary  $\partial\Omega$. Then there exists an $\varepsilon>0$ such that $\partial \Omega$ admits a tubular neighbourhood $N_{\varepsilon}(\partial \Omega)$ of radius $\varepsilon$, i.e.  $N_{\varepsilon}(\partial \Omega):=\big\{z\in \mathbb C^n: \delta(z)<\varepsilon\big\}$, and  the \textit{signed distance function} $\delta^{\ast}:\mathbb C^n\rightarrow \mathbb R$ defined by
$$\delta^{\ast}(z)=
\left\{
\begin{array}{lll}
-\delta(z)\quad  \mbox{for} \quad z\in \Omega
\\
\ \; \,  \delta(z)\quad  \mbox{for} \quad z\in \mathbb C^n\setminus\Omega
\end{array}
\right.
$$
is  $C^2$-smooth on $N_{\varepsilon}(\partial \Omega)$ with gradient $\nabla \delta^{\ast}$ of length one  and hence is  a defining function for $\Omega$ (cf. \cite[pp. 70--71]{Fri-Gra}; also \cite{Krantz-Parks}). Moreover, for every
$z\in N_{\varepsilon}(\partial \Omega)$,
there is a \textit{unique} boundary point $\pi(z)\in\partial \Omega$ at minimum distance from $z$, and $z\mapsto \pi(z)$ is $C^1$-smooth on $N_{\varepsilon}(\partial \Omega)$. Indeed, $\pi$ is explicitly given by $$\pi(z)=z-\delta^{\ast}(z)\nabla \delta^{\ast}(z),\qquad \forall\, z\in N_{\varepsilon}(\partial \Omega),$$
 and is called \textit{the projection along the outward normal direction}, in view of the fact that $$\nabla \delta^{\ast}(z)=\nabla \delta^{\ast}\circ\pi(z)$$
  for all $z\in N_{\varepsilon}(\partial \Omega)$ (cf.\cite[Lemma 2.1]{Balogh}). For each $z\in N_{\varepsilon}(\partial \Omega)$ and each $v\in T_z^{(1,0)}\Omega\cong \mathbb C^n$, we can decompose $v$ into $v=v_N(z)+v_T(z)$, where $$v_N(z)=\langle v, \nabla \delta^{\ast}(z)\rangle \nabla \delta^{\ast}(z)$$
   is the normal component of $v$ at $z$ and $v_T(z)=v-v_N(z)$ the tangential component of $v$ at $z$, and $\langle\, , \,\rangle$ denotes the standard Hermitian inner product on $\mathbb C^n$ given by (\ref{Hermitian product}). The following lemma was first proved in \cite[Lemma 1.2]{Abate12} for the Kobayashi metric, and the same argument shows that it also holds for the Carath\'{e}odory metric. Here we provide the detailed proof for the sake of completeness.
\begin{lemma}\label{rough estimate}
Let $\Omega\subset \mathbb C^n$ be a bounded domain with $C^2$ boundary. Then there exist  $\varepsilon>0$ and  $C=C(\varepsilon)>0$ such that
\begin{equation}\label{upper estimate}
 \mathcal{C}_{\Omega}(z, v)\leq\frac{|v_N(z)|}{\delta(z)}+C\frac{|v_T(z)|}{\delta(z)^{1/2}}
\end{equation}
for all $z\in \Omega$ with $\delta(z)<\varepsilon$ and $v\in\mathbb C^n$.
\end{lemma}
\begin{proof}
Let $\varepsilon>0$ be such that $\partial \Omega$ admits a tubular neighbourhood $N_{2\varepsilon}(\partial \Omega)$ of radius $2\varepsilon$. Then for every point $z\in  \Omega$ with $\delta(z)<\varepsilon$, there exists a $z_0\in \Omega$ such that $\delta(z_0)=\varepsilon$, $\pi(z_0)=\pi(z)$ and $\delta(z)=\varepsilon-|z-z_0|$. Thus the Euclidean ball $B$ of center $z_0$ and radius $\varepsilon$ is internally tangent to the boundary $\partial \Omega$ at $\pi(z_0)$. Comparing the Carath\'{e}odory metric on $\Omega$ and that on the ball $B$, and using the length-decreasing property yield that
$$\mathcal{C}_{\Omega}(z, v)\leq \mathcal{C}_B(z, v).$$
Therefore to finish the proof, it suffices to estimate the latter. Thanks to the explicit formula for  $\mathcal{C}_{\mathbb B^n}$, for every $v\in\mathbb C^n$, we have
\begin{equation*}\label{computation}
\begin{split}
\mathcal{C}^2_B(z, v)&=\frac{|v|^2}{\varepsilon^2-|z-z_0|^2}\
+\frac{\big|\langle v, z-z_0\rangle\big|^2}{\big(\varepsilon^2-|z-z_0|^2\big)^2} \\
&=\frac{\varepsilon^2}{\big(\varepsilon+|z-z_0|\big)^2}\frac{|v_N(z)|^2}{\delta^2(z)}
+\frac{|v_T(z)|^2}{(\varepsilon+|z-z_0|)\delta(z)}\\
&\leq \frac{|v_N(z)|^2}{\delta^2(z)}+\frac{|v_T(z)|^2}{\varepsilon\delta(z)}\\
&\leq \bigg(\frac{|v_N(z)|}{\delta(z)}+\varepsilon^{-1/2}\frac{|v_T(z)|}{\delta(z)^{1/2}}\bigg)^2,
\end{split}
\end{equation*}
and the desired estimate follows with $C(\varepsilon)=\varepsilon^{-1/2}$.
\end{proof}

In the proof of main result (Theorem \ref{main result-Schwarz}), we will also make use part of Graham's estimate  on the boundary behavior of the Carath\'{e}odory metric on strongly pseudoconvex domains. For the reader's convenience, we state here that what we need from \cite[Theorems 1 and 1']{Graham}:

\begin{lemma}[Graham]\label{Graham}
Let $\Omega\subset \subset \mathbb C^n$ be a strongly pseudoconvex domain with $C^2$ boundary and $\mathcal{L}_{\Omega}$ the Levi form of $\Omega$. Then
\begin{equation}\label{Graham estimate1}
\lim_{z\rightarrow p}\mathcal{C}^2_{\Omega}\big(z, \tau_p\big)\delta(z)
=\frac12\mathcal{L}_{\Omega}\big(p, \tau_p\big),\qquad \forall\, \tau_p\in T^{(1,0)}_{p}(\partial \Omega).
\end{equation}
Moreover,
\begin{equation}\label{Graham estimate2}
\lim_{z\rightarrow p}\mathcal{C}^2_{\Omega}\big(z, v_T(z)\big)\delta(z)
=\frac12\mathcal{L}_{\Omega}\big(p, v_T(p)\big),
\end{equation}
uniformly in $p\in\partial \Omega$ and $v\in\partial\mathbb B^n$.
\end{lemma}

%Let $D\subset \mathbb C^n$ be a bounded pseudoconvex domain. We denote by $k_D(z, w)$ the Carath\'{e}odory  distance of $z,w\in D$, and by $\kappa_D(z; v)$ the Carath\'{e}odory length in $D$ of the complex tangent vector $v\in T^{(1,0)}_z(D)$.

\section{Proof of main result}\label{Proof of main result}

In this section, we give a proof of  main result of this paper. We shall consistently use $C$ to denote positive constants, which could be different in different appearances.
\begin{proof}[Proof of Theorem $\ref{main result-Schwarz}$]
The proof is divided into five steps in order to make it easier to follow.
\bigskip
  \item [\textsf{Step} 1.] First, we show that both the real tangent space $T_{p}(\partial \Omega)$ and the complex tangent space $T_{p}^{(1,0)}(\partial \Omega)$ are invariant under the action of $J_f(p)$ as linear transformations on $\mathbb R^{2n}$ and $\mathbb C^n$, respectively. Namely,
      \begin{equation}\label{real invariant}
       df_p\big(T_{p}(\partial \Omega)\big)\subseteq T_{p}(\partial \Omega)
      \end{equation}
      and
      \begin{equation}\label{complex invariant}
       df_p\big(T_{p}^{(1,0)}(\partial \Omega)\big)\subseteq T_{p}^{(1,0)}(\partial \Omega).
      \end{equation}
      To this end, it is enough to show that $(\ref{real invariant})$ holds, since then $(\ref{complex invariant})$ immediately follows from the very definition of complex tangent space $T_{p}^{(1,0)}(\partial \Omega)$ and from the fact that $df_p=\partial f_p$ is $\mathbb C$-linear.

      Let $\nu_p$ be the unit outward normal vector to $\partial \Omega$ at the point $p\in\partial \Omega$ and $\rho$ a $C^2$-defining function for $\Omega$ such that $\rho$ is strictly plurisubharmonic  on some neighborhood of $\overline{\Omega}$, which always exists as guaranteed by the strong pseudoconvexity of $\Omega$. For each unit tangent vector $\tau\in T_{p}(\partial \Omega)$, we can take a $C^1$-curve $\gamma$ such that
      $\gamma\big((-1,1)\setminus\{0\}\big)\subset \Omega$, $\gamma(0)=p$ and $\gamma'(0)=\tau\in T_{p}(\partial \Omega)$, then the function
      $(-1,1)\ni t\mapsto \rho\circ f\circ\gamma(t)$
      has its maximum at $t=0$ and therefore
      $$d\rho_p\big(df_p(\tau)\big)=\left.\bigg(\frac{d}{dt}\rho\circ f\circ\gamma(t)\bigg)\right|_{t=0}=0,$$
      which implies that $df_p(\tau)\in T_{p}(\partial \Omega)$, and the desired result immediately follows.

       Moreover, the function $\rho\circ f$ is a negative plurisubharmonic (and hence subharmonic) function on $\Omega$ attaining its maximum at the boundary point $p\in\partial \Omega$. Therefore, the directional derivative of $\rho\circ f$ along $\nu_p$ at the point $p$ satisfies that
      \begin{equation}\label{directional derivative}
      d\rho_p\big(df_p(\nu_p)\big)
      =\frac{\partial (\rho\circ f)}{\partial \nu_p}(p)>0,
      \end{equation}
      in virtue of the classical Hopf  lemma for subharmonic functions.

\bigskip
  \item [\textsf{Step} 2.]
      What we just proved in Step 1 is equivalent to that
       both the real normal space $\mathbb R\nu_p$ at $p$   and the complex one $\mathbb C\nu_p$ at $p$  are invariant under the action of $\overline{J_f(p)}^t$ as linear transformations of $\mathbb R^{2n}$ and $\mathbb C^n$, respectively.
      Therefore $\overline{J_f(p)}^t$ must admit a real number $\lambda\in\mathbb R$ such that
      \begin{equation*}\label{lambda-relation}
      \overline{J_f(p)}^t\nu_p=\lambda\nu_p,
      \end{equation*}
      which together with $(\ref{directional derivative})$ gives that
      \begin{equation}\label{lambda-positive}
      \lambda={\rm{Re}}\Big\langle\nu_p, \overline{J_f(p)}^t\nu_p\Big\rangle
      ={\rm{Re}}\Big\langle J_f(p)\nu_p, \nu_p\Big\rangle
      =d\rho_p\big(df_p(\nu_p)\big)>0.
      \end{equation}
      Obviously, $\lambda>0$ is also an eigenvalue of $J_f(p)$.
\bigskip
\item [\textsf{Step} 3.]
      The assertion ${\rm{(iii)}}$ easily follows from what we proved in Step 1. If we proved the assertion ${\rm{(ii)}}$, then the assertion ${\rm{(iv)}}$ immediately follows.

      To prove the assertion ${\rm{(ii)}}$, we will make the best of Lemmas \ref{rough estimate} and \ref{Graham}. For $j=2, 3,\ldots, n$, each complex number $\mu_j\in\mathbb C$ and each vector $\tau_j\in T_{p}^{(1,0)}(\partial \Omega)\cap\partial\mathbb B^n$ satisfying
       \begin{equation}\label{eigenvalue}
       J_f(p)\tau_j=\mu_j\tau_j,
       \end{equation}
       it follows from the Schwarz inequality (\ref{Caratheodory-Schwarz}) that
      $$\mathcal{C}_{\Omega}\big(f(z), J_f(z)\tau_j\big)\leq \mathcal{C}_{\Omega}(z, \tau_j),$$
      and hence
      \begin{equation}\label{Caratheodory-Schwarz-ineq}
      \mathcal{C}_{\Omega}^2\big(f(z), J_f(z)\tau_j\big)\delta\big(f(z)\big)\leq \mathcal{C}^2_{\Omega}(z, \tau_j)\delta(z)\frac{\delta\big(f(z)\big)}{\delta(z)},\qquad \forall\, z\in \Omega.
      \end{equation}
      Since $\partial \Omega$ is of class $C^2$,   we may assume that
      $\delta$ is $C^2$-smooth on $\Omega\cap N_{\varepsilon}(\partial \Omega)$, where $N_{\varepsilon}(\partial \Omega)$ is a tubular neighbourhood  of $\partial \Omega$ with  radius $\varepsilon>0$
      taken as in the paragraph before Lemma \ref{rough estimate}. We shall also use the notations fixed there. Now it is transparent that
     \begin{equation}\label{boundary limit}
      \lim_{z\rightarrow p}\frac{\delta\big(f(z)\big)}{\delta(z)}={\rm{Re}}\Big\langle J_f(p)\nu_p, \nu_p\Big\rangle=\lambda>0.
      \end{equation}

      We next turn to the estimate in assertion ${\rm{(ii)}}$. First of all, equality $(\ref{Graham estimate1})$ shows that
      \begin{equation}\label{tangential limit1}
       \lim_{z\rightarrow p}\mathcal{C}^2_{\Omega}(z, \tau_j)\delta(z)=\frac12\mathcal{L}_{\Omega}\big(p, \tau_j\big)>0.
      \end{equation}
      Obviously, it suffices to consider the $j$'s satisfying the condition that $\mu_j\neq 0$ in assertion ${\rm{(ii)}}$. For each such $j$, we claim that
      \begin{equation}\label{tangential limit2}
      \lim_{z\rightarrow p}\mathcal{C}_{\Omega}^2\big(f(z), J_f(z)\tau_j\big)\delta\big(f(z)\big)=\frac12\mathcal{L}_{\Omega}\big(p, J_f(p)\tau_j\big),
      \end{equation}
      from which and $(\ref{eigenvalue})$--$(\ref{tangential limit1})$ the estimate in assertion ${\rm{(ii)}}$ will follow.
      In virtue of the limit in $(\ref{Graham estimate2})$ and its uniform convergence, we obtain that
      \begin{equation}\label{tangential limit3}
      \begin{split}
       \lim_{z\rightarrow p}\mathcal{C}_{\Omega}^2\Big(f(z), \big(J_f(z)\tau_j\big)_T\big(f(z)\big)\Big)\delta\big(f(z)\big)
       =&\frac12\mathcal{L}_{\Omega}\Big(f(p), \big(J_f(p)\tau_j\big)_T\big(f(p)\big)\Big)\\
       = &\frac12\mathcal{L}_{\Omega}\big(p, J_f(p)\tau_j\big).
      \end{split}
      \end{equation}
      For each $z\in \Omega\cap N_{\varepsilon}(\partial \Omega)$, it follows from the subadditivity of   $\mathcal{C}_{\Omega}(z, v)$ in $v\in T_z^{(1,0)}\Omega$ that
      \begin{equation*}
      \begin{split}
      &\mathcal{C}_{\Omega}\Big(f(z), \big(J_f(z)\tau_j\big)_T\big(f(z)\big)\Big)
      -\mathcal{C}_{\Omega}\Big(f(z), \big(J_f(z)\tau_j\big)_N\big(f(z)\big)\Big)\\
      \leq &\,\mathcal{C}_{\Omega}\big(f(z), J_f(z)\tau_j\big)\\
      \leq &\,\mathcal{C}_{\Omega}\Big(f(z), \big(J_f(z)\tau_j\big)_T\big(f(z)\big)\Big)+\mathcal{C}_{\Omega}\Big(f(z), \big(J_f(z)\tau_j\big)_N\big(f(z)\big)\Big).
        \end{split}
      \end{equation*}
      Also from the strong pseudoconvexity of $\Omega$ and equality $(\ref{tangential limit3})$ it follows that
      $$\mathcal{C}_{\Omega}\Big(f(z), \big(J_f(z)\tau_j\big)_T\big(f(z)\big)\Big)$$ is non-vanishing whenever $z\in {\Omega}$ is sufficiently close to $p$.
      Now to prove equality $(\ref{tangential limit2})$, it suffices to show that
      \begin{equation}\label{quoient limit}
      \lim_{z\rightarrow p}\frac{\mathcal{C}_{\Omega}\Big(f(z), \big(J_f(z)\tau_j\big)_N\big(f(z)\big)\Big)}{\mathcal{C}_{\Omega}\Big(f(z), \big(J_f(z)\tau_j\big)_T\big(f(z)\big)\Big)}=0.
      \end{equation}
      Indeed, from inequality $(\ref{upper estimate})$, equalities $(\ref{boundary limit})$  and $(\ref{tangential limit3})$ we conclude easily that
      \begin{equation}\label{compare-ineq}
      \frac{\mathcal{C}_{\Omega}\Big(f(z), \big(J_f(z)\tau_j\big)_N\big(f(z)\big)\Big)}{\mathcal{C}_{\Omega}\Big(f(z), \big(J_f(z)\tau_j\big)_T\big(f(z)\big)\Big)}
      \leq C
      \frac{\big|\big(J_f(z)\tau_j\big)_N\big(f(z)\big)\big|}{\delta(z)^{1/2}}
      \leq C  \delta(z)^{1/2}
      \end{equation}
      for some constant $C>0$.
      In the last step we have used the obvious fact that $J_f(z)-J_f(p)=O(|z-p|)$ as $z\rightarrow p$. Now  $(\ref{quoient limit})$ follows and so does $(\ref{tangential limit2})$.

\bigskip
  \item [\textsf{Step} 4.] Now we claim that $\lambda, \mu_2, \ldots, \mu_n$ are all eigenvalues of the linear transformation $J_f(p): \mathbb C^n\rightarrow \mathbb C^n$.
      The argument is the same as that in \cite{LWX}. We provide the details for completeness. Assume that $\big\{v_2, v_3, \ldots, v_n\big\}$ is an orthonormal basis for $T_{p}^{(1,0)}(\partial {\Omega})$ with respect to the standard Hermitian inner product on $T_{p}^{(1,0)}(\partial {\Omega})\cong \mathbb C^{n-1}$, and hence $\big\{\nu_p, v_2, \ldots, v_n\big\}$ is an orthonormal basis for $\mathbb C^n$. Set
      $$U=\big(\nu_p, v_2, \ldots, v_n\big),$$
      then $U$ is a unitary matrix of order $n$. Since $\overline{J_f(p)}^t\nu_p=\lambda\nu_p,$ we obtain that
      $$\overline{J_f(p)}^tU=U\left(
                                \begin{array}{cc}
                                  \lambda & \overline{B}^t \\
                                  0& \overline{V}^t\\
                                \end{array}
                              \right),
      $$
      where $V$ is a complex matrix of order $(n-1)$, and $B$ is a $(n-1)\times 1$ complex matrix. The previous equality is equivalent to
      $$J_f(p)U=U\left(
      \begin{array}{cc}
      \lambda & 0\\
      B& V\\
      \end{array}
      \right).
      $$
      That is,
      $$J_f(p)\big(\nu_p, v_2, \ldots, v_n\big)
      =\big(\nu_p, v_2, \ldots, v_n\big)\left(
      \begin{array}{cc}
      \lambda & 0\\
      B& V\\
      \end{array}
      \right),$$
      which implies that
      \begin{equation}\label{V-transform}
      \left.J_f(p)\right|_{T_{p}^{(1,0)}(\partial \Omega)}\big(v_2, v_3, \ldots, v_n\big)
      =\big(v_2, v_3, \ldots, v_n\big)V,
      \end{equation}
      the roots of the characteristic polynomial $\det(tI_{n-1}-V)$ are exactly $\mu_2, \mu_3,\ldots, \mu_n$. If $\lambda\notin\big\{\mu_2, \mu_3,\ldots, \mu_n\big\}$, then $\lambda, \mu_2, \ldots, \mu_n$  are all eigenvalues of the linear transformation $J_f(p): \mathbb C^n\rightarrow \mathbb C^n$. Otherwise, we assume that $\lambda=\mu_{j_0}$, where $\mu_{j_0}$ is a zero of  $\det(tI_{n-1}-V)$ of order $k\in\mathbb N^{\ast}$, the obvious fact that the characteristic polynomial of $J_f(p)$ is
      $$\det\bigg(tI_n-\left(
              \begin{array}{cc}
                \lambda & 0 \\
                B & V \\
              \end{array}
            \right)\bigg)=(t-\lambda)\det(tI_{n-1}-V)
      $$
      shows that $\lambda=\mu_{j_0}$ is an eigenvalue of order $k+1$ of $J_f(p): \mathbb C^n\rightarrow \mathbb C^n$.

\bigskip
\item [\textsf{Step} 5.] We claim that if $f$ has an interior fixed point $z_0\in {\Omega}$, then $\lambda\geq 1.$ From Step 4, it is easy to see that $\lambda$ is also a singular value of $J_f(p)$. Unfortunately, it seems that we can not obtain an estimate on the upper bounds of  the other singular values of  $J_f(p)$ in terms of $\lambda$ as in \cite[Proposition 1.1]{Bracci4}.
    Suppose by contradiction that $\lambda<1$ under the extra condition that $z_0\in {\Omega}$ is an interior fixed point of $f$, and consider the sequence $\{f^k\}$ of iterates of $f$. From ${\rm{(ii)}}$, we know that all the eigenvalue values of $J_f(p)$ are strictly less than one. Taking into account the Jordan canonical form of $J_f(p)$, we deduce that there exists a sufficiently large $k_0\in\mathbb N$ such that $(J_f(p))^{k_0}:\mathbb C^n\rightarrow\mathbb C^n$ is a contraction. Therefore for every point $z\in \Omega$ sufficiently close to $p$, the sequence $\{f^{kk_0}(z)\}$ converges to $p$. Moreover, since $\Omega$ is strongly pseudoconvex, $p\in\partial \Omega$ is a peak point for $\Omega$, it follows from \cite[Propostion 2.3.59]{Abate} that the sequence $\{f^{kk_0}\}$ converges to $p$, contradicting the fact that $f^{kk_0}(z_0)$ equals to $z_0$ for all $k\in \mathbb{N}$. This completes the proof.
\end{proof}

Here we state the result obtained in Step. 5 above as the following independent proposition, since it seems to us that it will be useful in other similar situations.

\begin{proposition}
Let $f: \Omega\rightarrow \Omega$ be a holomorphic self-mapping of   strongly pseudoconvex domain $\Omega\subset \subset\mathbb C^n$ with $C^2$ boundary. If $f$ can extend smoothly to some point $p\in\partial \Omega$, and all the eigenvalue values of the Jacobian $J_f(p)$ of $f$ at $p$ are strictly less than one, then there is a number  $k_0\in\mathbb{N}$ such that the sequence $\{f^{kk_0}\}$ of iterates of $f^{k_0}$ converges locally uniformly to $p$.
\end{proposition}

\begin{proof}
We have proved in Step. 5 above the existence of $k_0\in\mathbb{N}$ with the desired property. Since  $\Omega$ is bounded, the local uniformity of  this convergence follows easily from the classical Montel theorem, which implies that for each domain $D\subset \mathbb C^n$, the topology of pointwise convergence on $\textrm{Hol}(D, \Omega)$ coincides with the usual compact-open topology.
\end{proof}

Some useful remarks concerning Theorem \ref{main result-Schwarz} are in order.

\begin{remark}
In the preceding proof, we can also replace the Carath\'{e}odory metric by the Kobayashi metric. The choice of the Carath\'{e}odory metric is convenient to prove $(\ref{tangential limit3})$ using the fact that for each $z\in \Omega$,  $\mathcal{C}_{\Omega}(z, v)$ is subadditive in $v\in T_z^{(1,0)}\Omega$. Even though the
subadditivity property fails in general for the Kobayashi metric, equalities in $(\ref{tangential limit3})$ also hold with the Carath\'{e}odory metric replaced by the Kobayashi metric. This follows by checking carefully the origin proof of Graham's estimate in \cite{Graham}. Therefore, our argument also works for the Kobayashi metric.
\end{remark}

\begin{remark}
From our argument in Step. 4 above, we conclude that (as Bracci pointed out to the author) either there exists a vector $v\in \mathbb C^n\setminus T^{(1,0)}(\partial \Omega)$ such that $J_f(p)v=\lambda v$, or there exists $2\leq m\leq n$ such that $\mu_m=\lambda$ and the geometric multiplicity of the eigenvalue $\lambda$ is strictly less than its algebraic multiplicity. Moreover, the second case is possible if and only if $\lambda=1$.
\end{remark}

\begin{remark}
In the case that $\Omega$ is a bounded strongly convex domain with  $C^3$ boundary, the fact that $\lambda\geq 1$ provided that $f$ has an interior fixed point $z_0\in \Omega$ was already known. Indeed, from \cite[Theorem 2.6.47]{Abate} and its proof one can easily deduce that the eigenvalue $\lambda$ of $f_f(p)$ is no other than the \textit{boundary dilatation coefficient} $\alpha_f(p)$ of $f$ at the point $p\in\partial \Omega$, defined by means of
\begin{equation}\label{BD-Coefficient}
\frac12 \log\alpha_f(p)=\liminf_{z\rightarrow p}\Big(k_{\Omega}(w_0, z)-k_{\Omega}\big(w_0, f(z)\big)\Big),
\end{equation}
which is independent of the base point $w_0\in \Omega$ whenever $p$ is a regular boundary fixed point of $f$ (see \cite[Lemma 6.1]{Bracci0}; also \cite[Lemma 1.3]{Abate13}), which is the case under our assumption. Here $k_{\Omega}$ is the Kobayashi distance of $\Omega$. In particular, setting $w_0=z_0$ in the right-hand side of  $(\ref{BD-Coefficient})$ and noticing that $k_{\Omega}\big(z_0, f(z)\big)=k_{\Omega}\big(f(z_0), f(z)\big)\leq k_{\Omega}(z_0, z)$ yield that $\lambda=\alpha_f(p)\geq1$. Moreover, when $\lambda=1$, one can obtain more information about $f$, see \cite[Theorem 2.4]{Bracci-CD} for details.

\end{remark}

\begin{remark}

In some sense, Theorem \ref{main result-Schwarz} is related to Abate's generalization of the classical Julia-Wolff-Carath\'{e}odory theorem, see  \cite[Theorem 0.2]{Abate11} for details. The theory of Lempert's complex geodesics plays a central role in Abate's argument,  which requires the boundary of domain $\Omega$ to be of class $C^3$, while in our argument Graham's estimate of Carath\'{e}odory metric on strongly pseudoconvex domains is a fundamental tool and $C^2$-regularity of the boundary $\partial \Omega$ is enough.
\end{remark}

%\begin{remark}
%The previous argument also gives an estimate (which was also obtained in \cite[Propostion 1.1]{Bracci4}) on the singular values $\sigma_1, \sigma_2,\ldots, \sigma_{n-1}\geq 0$  of the matrix $V$ appeared in $(\ref{V-transform})$:
%$$\sigma:=\max\big\{\sigma_1, \sigma_2,\ldots, \sigma_{n-1}\big\}\leq \sqrt \lambda.$$
%Indeed, from the very definition of singular values of a matrix and  $(\ref{V-transform})$ it follows that $\sigma^2$ is the maximal eigenvalue of the linear transformation
%$$\left.J_f(p)\right|_{T_{p}^{(1,0)}(\partial D)}\left.J_f(p)\right|^{\ast}_{T_{p}^{(1,0)}(\partial D)}: J_f(p)\big(T_{p}^{(1,0)}(\partial D)\big)\rightarrow J_f(p)\big(T_{p}^{(1,0)}(\partial D)\big).$$
%Let $v\in J_f(p)\big(T_{p}^{(1,0)}(\partial D)\big)\subseteq T_{p}^{(1,0)}(\partial D)$ be an eigenvector of this transformation with respect to $\sigma^2$, i.e.
%$$J_f(p)\left.J_f(p)\right|^{\ast}_{T_{p}^{(1,0)}(\partial D)}(v)=\sigma^2 v.$$
%Replacing $\tau_j$ by $\left.J_f(p)\right|^{\ast}_{T_{p}^{(1,0)}(\partial D)}(v)$ and repeating the previous argument yield that
%
%\end{remark}

\begin{remark}\label{weak condition2}
We can weaken the assumption of regularity of $f$ at the point $p$, and $C^2$-regularity is enough at least in the previous argument. We can also consider the case that the boundary point $p\in\partial \Omega$ is only a regular boundary contact point of $f$ and even the case that $f:D\rightarrow\Omega$ is a holomorphic mapping between two different strongly pesudoconvex domains $D\subset\mathbb C^m$, $\Omega\subset\mathbb C^n$. The previous argument still works well. We have chosen to restrict to the case as in Theorem \ref{main result-Schwarz}  in order to emphasize the basic ideas.
\end{remark}

As a direct consequence of Theorem \ref{main result-Schwarz}, we have the following result (it is exactly \cite[Theorem 3.1]{LWX} except ${\rm{(vi)}}$ and  the first equality in ${\rm{(i)}}$):

\begin{corollary}\label{Schwarz in C^n}
Let $f:\mathbb B^n\rightarrow\mathbb B^n$  be a holomorphic mapping. If $f$ is holomorphic at $p\in\partial \mathbb B^n$ and $f(p)=p$, then for the eigenvalues $\lambda, \mu_2,\ldots,\mu_n$ $($counted with multiplicities$)$ of $J_f(p)$, the following statements hold:
 \begin{enumerate}
 \item [(i)] $\lambda=\dfrac{\partial |f|}{\partial p}\geq \dfrac{|1-\langle f(0), p\rangle|^2}{1-|f(0)|^2}>0$;

 \item [(ii)] $p$ is  an eigenvalue of $\overline{J_f(p)}^t$ with respect to $\lambda$, that is $\overline{J_f(p)}^t p=\lambda p$;

 \item [(iii)] $\mu_j\in\mathbb C$ and $|\mu_j|\leq\sqrt{\lambda}$ for $j=2,\ldots,n$;

 \item [(iv)] For any $\mu_j$, there exists $\tau_j\in\partial \mathbb B^n\cap T_{p}^{(1,0)}(\partial \mathbb B^n)$ such that
    \begin{equation*}
    J_f(p)\tau_j=\mu_j\tau_j,\qquad \forall\, j=2,\ldots,n;
    \end{equation*}
\item [(v)] $|\det J_f(p)|\leq\lambda^{\frac{n+1}2},\qquad
|{\rm{tr}} J_f(p)|\leq\lambda+(n-1)\sqrt{\lambda}$;
\item [(vi)]
\begin{equation}\label{Julia-ineq for Ball}
\frac{\partial |f|}{\partial p}f^{\ast}\Omega_{\mathbb B^n,\,p}\leq\Omega_{\mathbb B^n,\,p},
\end{equation}
    where
    $$\Omega_{\mathbb B^n,\,p}(z)=-\frac{1-|z|^2}{\big|1-\langle z, p\rangle\big|^2}$$
    is the $(negative)$ pluricomplex Poisson kernel of $\mathbb B^n$ with a simple singularity at $p$; and equality holds in $(\ref{Julia-ineq for Ball})$ if and only if $f\in Aut(\mathbb B^n)$ is such that $f(p)=p$.
 \end{enumerate}

 Moreover, the inequalities in ${\rm{(i)}}$,  ${\rm{(iii)}}$ and ${\rm{(v)}}$ are sharp.
\end{corollary}
\begin{proof}
It suffices to prove inequality $(\ref{Julia-ineq for Ball})$ and  the first equality in ${\rm{(i)}}$. The latter follows directly from equality $(\ref{boundary limit})$. Inequality $(\ref{Julia-ineq for Ball})$ is nothing but the classical Julia inequality (see \cite[Theorem 2.2.21]{Abate}), which was first rephrased by Bracci et al. in \cite{Bracci3} as the form of $(\ref{Julia-ineq for Ball})$.  As for the condition for equality, see \cite[Theorem 7.3]{Bracci2}.
\end{proof}

\begin{remark}

Despite the estimate in ${\rm{(iii)}}$, which is due to Schwarz lemma, and the condition for equality in $(\ref{Julia-ineq for Ball})$, Corollary \ref{Schwarz in C^n} is essentially a direct consequence of Rudin's generalization of the classical Julia-Wolff-Carath\'{e}odory theorem (see \cite[Theorem 8.5.6]{Rudin} or \cite[Theorem 2.2.29]{Abate}) under the extra assumption of regularity of $f$ at $p\in\partial \mathbb B^n$.

\end{remark}

\section*{Acknowledgements}

The first author is very grateful to Professor Filippo Bracci for bringing the paper \cite{Bracci4} into his attention, and several valuable communications and comments, from  which he has learnt a lot.  The authors would like to thank the anonymous referee for his/her careful reading of this paper.

\bigskip
\bibliographystyle{amsplain}

\begin{thebibliography}{99}
\bibitem{Abate} M. Abate, \textit{Iteration Theory of Holomorphic Maps on Taut Manifolds}. Mediterranean Press, Rende, 1989. See also \url{http://www.dm.unipi.it/~abate/libri/libriric/libriric.html.}
\bibitem{Abate11} M. Abate, \textit{Angular derivatives in strongly pseudoconvex domains}. Proc. Symp. Pure Math. 52, Part 2, (1991), 23--40.
\bibitem{Abate12} M. Abate, R. Tauraso, \textit{The Lindel\"{o}f principle and angular derivatives in convex domains of finite type}. J. Aust. Math. Soc. \textbf{73} (2002), 221--250.

\bibitem{Abate13} M. Abate, J. Raissy, \textit{Backward iteration in strongly convex domains}. Adv. Math. \textbf{228} (2011), 2837--2854.



\bibitem{Balogh} Z. M. Balogh, M. Bonk, \textit{Gromov hyperbolicity and the Kobayashi metric on strictly pseudoconvex domains}. Comment. Math. Helv. \textbf{75} (2000), 504--533.
\bibitem{BZZ} L. Baracco, D. Zaitsev, G. Zampieri, \textit{A Burns-Krantz type theorem for domains with corners}. Math. Ann. \textbf{336} (2006), 491--504.


\bibitem{Bracci0} F. Bracci, \textit{Dilatation and order of contact for holomorphic self-maps of strongly convex domains}. Proc. London Math. Soc. \textbf{86} (2003), 131--152.

\bibitem{Bracci-CD} F. Bracci, M. Contreras, S. D\'{\i}az-Madrigal, \textit{Pluripotential theory, semigroups and boundary behavior of infinitesimal generators in strongly convex domains}. J. Eur. Math. Soc. \textbf{12} (2010), 23--53.

\bibitem{Bracci01} F. Bracci, J. E. Forn{\ae}ss, \textit{The range of holomorphic maps at the boundary points}. Math. Ann. \textbf{359} (2014), 909--927.

\bibitem{Bracci2} F. Bracci, G. Patrizio, \textit{Monge-Amp\`{e}re foliations with singularities at the boundary of strongly  convex domains}. Math. Ann. \textbf{332} (2005), 499--522.


\bibitem{Bracci3} F. Bracci, G. Patrizio, S. Trapani, \textit{The pluricomplex Poisson kernel for strongly convex domains}. Trans. Amer. Math. Soc. \textbf{361} (2009), 979--1005.

\bibitem{Bracci4} F. Bracci, D. Zaitsev, \textit{Boundary jets of holomorphic maps between strongly pseudoconvex domains}. J. Funct. Anal. \textbf{254} (2008), 1449--1466.

\bibitem{BK} D. M. Burns, S. G. Krantz, \textit{Rigidity of holomorphic mappings and a new Schwarz lemma at the boundary}. J. Amer. Math. Soc. \textbf{7} (1994), 661--676.


\bibitem{EJLS} M. Elin, F. Jacobzon, M. Levenshtein, D. Shoikhet, \textit{The Schwarz lemma: rigidity and dynamics.} Harmonic and Complex Analysis and its Applications. Springer International Publishing, 2014, pp. 135--230.


\bibitem{Fri-Gra} K. Fritzsche, H. Grauert, \textit{From Holomorphic Functions to Complex Manifolds}. Graduate Texts in Mathematics, 213. Springer-Verlag, New York, 2002.
\bibitem{Graham} I. Graham, \textit{Boundary behavior of the Carath\'{e}odory and Kobayashi metrics on strongly pseudoconvex domains in $\mathbb C^n$ with smooth boundary}. Trans. Amer. Math. Soc. \textbf{207} (1975), 219--240.

\bibitem{Herzig} A. Herzig, \textit{Die Winkelderivierte und das  Poisson-Stieltjes-Integral}. Math. Z. \textbf{46} (1940), 129--156.

\bibitem{Huang} X. Huang, \textit{A boundary rigidity problem for holomorphic mappings on some weakly pseudoconvex domains}. Can. J. Math. \textbf{47} (1995), 405--420.
\bibitem{J-Pflug} M. Jarnicki, P. Pflug, \textit{Invariant Distances and Metrics in Complex Analysis}. Walter de Gruyter, Berlin, 2013.

\bibitem{Kobayashi} S. Kobayashi, \textit{Hyperbolic Complex Spaces}. Springer, Berlin, 1998.

\bibitem{Krantz} S. G. Krantz, \textit{Function Theory of Several Complex Variables}. 2nd ed., Amer. Math. Soc., Providence, 2001.
\bibitem{Krantz-Parks} S. G. Krantz, H. R. Parks, \textit{Distance to $C^k$ hypersurfaces}. J. Differential Equations. \textbf{40} (1981), 116--120.

\bibitem{LWX} T. Liu, J. Wang, X. Tang, \textit{Schwarz lemma at the boundary of the unit ball in $\mathbb C^n$ and its applications}. J. Geom. Anal. \textbf{25} (2015), 1890--1914.
\bibitem{Osserman} R. Osserman, \textit{A sharp Schwarz inequality on the boundary}. Proc. Amer. Math. Soc. \textbf{128} (2000), 3513--3517.

\bibitem{WR}  G. Ren, X. Wang, \textit{Julia theory for slice regular functions}. Trans. Amer. Math. Soc. 2016, in press.

\bibitem{Rudin} W. Rudin, \textit{Function Theory in the Unit Ball of $\mathbb C^n$}, Springer-Verlag, New York, 1987.



\bibitem{Unkelbach} H. Unkelbach, \textit{\"{U}ber die Randverzerrung bei konformer Abbildung}. Math. Z. \textbf{43} (1938), 739--742.


\bibitem{Wang-Ren} X. Wang, G. Ren, \textit{Boundary Schwarz lemma and Blaschke products}. Preprint, 2015.
\bibitem{Wu} H. Wu, \textit{Normal families of holomorphic mappings}. Acta Math. \textbf{119} (1967), 193--233.
\end{thebibliography}

\end{document}